\documentclass{amsart}

\usepackage[english]{babel}
\usepackage{amsmath}

\newtheorem{propo}{Proposition}[section]

\newtheorem{lemma}[propo]{Lemma}

\newtheorem{theo}[propo]{Theorem}

\newtheorem{prob}[propo]{Problem}

\newcommand{\ra}{ \rightarrow }
\newcommand{\Ker}{\mathop{\rm Ker}\nolimits}
\newcommand{\ZZ}{\mathbb{Z}}
\newcommand{\G}{\mathbf{G}}
\newcommand{\Sym}{\mathbb{S}}
\newcommand{\Alt}{\mathbb{A}}
\newcommand{\Q}{\mathbb{Q}}
\newcommand{\lam}{\lambda } 

\newcommand{\ep}{\varepsilon}
\newcommand{\Irr}{\mathop{\rm Irr}\nolimits}
\newcommand{\IBr}{\mathop{\rm IBr}\nolimits}
\newcommand{\ld}{,\ldots , }

\begin{document}

\title[Characters of  simple
groups constant at the $p$-singular elements]{Irreducible characters of finite simple
groups constant at the $p$-singular elements}

\author{M.A. Pellegrini}
\address{Dipartimento di Matematica e Fisica, Universit\`a Cattolica del Sacro Cuore, Via Musei 41,
25121 Brescia, Italy}
\email{marcoantonio.pellegrini@unicatt.it}

\author[A. Zalesski]{A. Zalesski}
\address{Academy of Sciences of Belarus,
Minsk, Prospekt Nezalejnasti 66, Minsk, 220000, Belarus \and 
University of East Anglia, Norwich, University plain, NR4 7TJ, UK}
\email{alexandre.zalesski@gmail.com}

\keywords{Chevalley groups; Alternating groups; Irreducible characters; Principal block}
\subjclass[2010]{20C15, 20G40}

\begin{abstract}
In representation theory of finite groups an important role is played by irreducible 
characters of $p$-defect $0$, 
for a prime $p$ dividing the group order. These are exactly those vanishing 
at the $p$-singular elements.
In this paper we generalize this notion investigating the irreducible characters 
that  are constant at the $p$-singular elements. We determine all such characters 
of non-zero defect for alternating, symmetric and sporadic simple groups. 

We also  classify the irreducible characters  of quasi-simple groups of Lie type that are
constant at the non-identity unipotent elements.
In particular, we show that for groups of BN-pair rank greater than $2$ the Steinberg
and
the trivial characters are the only characters in question.
Additionally, we determine all irreducible characters whose degrees differ by $1$ from
the degree of the Steinberg character.
\end{abstract}

\maketitle

\section{Introduction}

Local representation theory studies properties of group representations  depending on 
a prime $p$ dividing the order of a finite group $G$ and the structure of a  Sylow $p$-subgroup $S$ of $G$.  
Denote by $\varSigma_p(G)$ the set of all $p$-singular elements of $G$, that is, those  of
order divisible by $p$. 
In this theory  a prominent role is played by irreducible characters
of defect $0$. These are exactly those vanishing at $\varSigma_p(G)$. 
In this paper we study   irreducible characters that are constant at $\varSigma_p(G)$. 
 We
call such characters \emph{$p$-constant}.

Although $p$-constant characters are very natural as a generalization of those of defect
$0$, they  do not seem to be discussed in the literature.

 If $G$ has a single conjugacy class of $p$-singular elements
then  every irreducible character of $G$ is $p$-constant. 
Groups $G$ with single  class of non-trivial $p$-elements are studied
in \cite{KN}.  
Also, the trivial character is
$p$-constant. It is less obvious that for $p>2 $
 non-exceptional characters in the principal  block with
cyclic defect are $p$-constant (see Theorem \ref{th2} below). We  mention \cite{NR} where
the
authors study irreducible characters
whose values at the $p$-singular elements are roots of unity, mainly for $p$-solvable
groups.

In this paper we focus mainly on quasi-simple groups and in view of Lemma
\ref{cov} below, we can concentrate on simple groups. Our main result is that on
classification 
of all $p$-constant irreducible characters for quasi-simple groups of Lie type with
defining characteristic $p$.  Following \cite[1.17]{Ca}, a finite group of Lie
type is the group of the fixed points of a (non-necessarily standard) Frobenius map acting
on a connected reductive group. (The simple group $^2F_4(2)'$ will be considered in Section \ref{Sspor} together with  
the sporadic groups).
Note that
among the quasi-simple groups of Lie type, only 
$SL_2(q)$ with $q$ even has a single class of non-identity $p$-elements.
Recall that, for every quasi-simple group $G$ of Lie type of characteristic $p$, the
Steinberg character is the only irreducible character of $G$ of $p$-defect $0$.

\begin{theo}\label{th1}
Let $G$ be a quasi-simple finite group of Lie type of characteristic $p$ and let $\tau$ be an irreducible character of $G$. 
Then $\tau$ is $p$-constant if, and only if, one of the following holds:
\begin{itemize}
\item[(1)] $\tau$ is the Steinberg character of $G$ or $\tau=1_G$;
\item[(2)] $G\in \{SL_2(q), SL_3(q), SU_3(q), {}^2B_2(q^2), {}^2G_2(q^2) \}$ and $\tau(1)=|G|_p\pm 1$.
More precisely, $\tau(1)\neq |G|_p-1$ if $G\in\{SU_3(q),{}^2B_2(q^2),{}^2G_2(q^2)\}$, and 
$\tau(1)\neq |G|_p+1$ if $G=SL_3(q)$.
\end{itemize}
\end{theo}

One can be interested with the other quasi-simple groups. We state the following.

\begin{prob}\label{prob1}
Let $G$ be a finite quasi-simple group. Determine 
the irreducible characters $\tau$ of $G$ for which there exists a constant $c\neq 0$ such
that $\tau(g)=c$ for all
$g\in \varSigma_p(G)$. 
\end{prob}

Toward this problem, we have the following technical but useful observations.
Recall that, when $G$ has cyclic Sylow $p$-subgroups, $\Irr G$ consists of so
called
exceptional and non-exceptional characters, see
\cite[Ch. VII]{Fe}. 

\begin{theo}\label{th2}
Let $G$ be a finite group with Sylow $p$-subgroup $S$,
and let $B$ be the principal $p$-block of $G$.
\begin{itemize}
\item[(1)]  If $\chi$ is an irreducible
$p$-constant character of non-zero defect, then $\chi$ belongs to $B$.
\item[(2)] Assume further that the  defect group $S$ of $B$ is cyclic and that $B$
contains $d$ ordinary exceptional characters.  Let $\chi\neq 1_G$ be an irreducible character
belonging
to $B$.  Then $\chi$   is $p$-constant if, and only if,  one of the following occurs:
\begin{itemize}
\item[\rm{(a)}] $d=1$;
\item[\rm{(b)}] $d>1$, $p>2$ and 
$\chi$ is not exceptional.
\end{itemize}
\noindent In addition, if $\chi$ is $p$-constant then  $\chi(g)=1$ or $-1$ for $g\in 
\varSigma_p(G)$. 
\end{itemize}
\end{theo}

In fact, we have more precise information on $\chi$ in the case (2) above in terms of the Brauer tree of the principal block. This reduces Problem \ref{prob1} to groups with non-cyclic Sylow
$p$-subgroups.  For alternating groups we have the following result:

\begin{theo}\label{th3} 
Let $G=\Alt_n$, $n>4$, be an alternating group, and let $p$ be a prime such  that $n\geq
2p$. Let
$\tau$ be a $p$-constant 
non-linear irreducible character of non-zero defect. Then one of the following holds:
\begin{itemize}
\item[(1)] $p>2$, $n=2p$ and $\tau$ is an irreducible constituent of an  irreducible
character of   $\Sym_n$ corresponding to the partition   $(p,1^{p})$ or $(p,2,1^{p-2})$;
\item[(2)] $p>2$, $n=2p+1$ and $\tau$ is an irreducible constituent of an  irreducible
character of   $\Sym_n$ corresponding to the partition   $(p+1,1^{p})$ or
$(p+1,2,1^{p-2})$;
\item[(3)] $p=2$ and $(n,\chi(1))\in\{(5,3), (5,5), (6,9),(7,15)\}$.
\end{itemize}
All these characters take value $1$ or $-1$ on $ \varSigma_p(G)$. 
\end{theo}

For finite simple groups we obtain the following result.

\begin{theo}\label{main}
Let $G$ be a finite simple group, $p$ be a prime dividing the order of $G$ and $\tau$ be
an irreducible character of $G$.
Assume that $\tau(g)=c$ for  all  $g\in \varSigma_p(G)$. Then, one of the
following holds:
\begin{itemize}
\item[(1)] $c\in \{-1,0,1\}$;
\item[(2)] $G=M_{22}$, $p=3$, $c=-2$ and  $\tau(1)=385$;
\item[(3)] $G$ is a group of Lie type
of characteristic $r\neq p$ with a non-cyclic Sylow $p$-subgroup.
\end{itemize}
\end{theo}

\noindent Note that  case (3) requires further analysis. This case is not vacuous: for instance the group $PSL_3(7)$ 
admits an irreducible $3$-constant character which takes value  $2$ at the $3$-singular elements.
For sporadic groups see
Section \ref{Sspor}. 

In \cite{Se} Seitz discussed a question on pairs of irreducible characters of classical groups whose degrees differ by $1$. 
He suggested examples, currently known as irreducible Weil characters, and studied these examples in certain details. 
To our knowledge, no further discussion of this question is available in the
literature
(but the Weil characters themselves  
attracted a lot of attention and have many applications). As a part of our proof of Theorem \ref{th1}
we classify all cases where one of the characters is the Steinberg character $St_G$ of
a finite group of Lie type $G$.

\begin{theo}\label{th4}
Let $G$ be a quasi-simple finite group of Lie type. Then $G$ admits  an irreducible character $\tau$  such that 
$\tau(1)=St_G(1)\pm 1$ if, and only if, one of the following holds:
\begin{itemize}
\item[(1)] $\tau(1)=St_G(1)+1$ and $G\in \{SL_2(q), SU_3(q), {}^2B_2(q^2), {}^2G_2(q^2)\}$.
\item[(2)]  $\tau(1)=St_G(1)-1$ and $G\in \{SL_2(q), SL_3(q), Sp_4(q), G_2(q)\}$. 
\end{itemize} 
\end{theo}

In Section \ref{sec2} we give some basic properties of $p$-constant characters and
we recall some results of E. Dade in order to prove Theorem \ref{th2}. In Section
\ref{sec3} we deal with symmetric and alternating groups. In Section \ref{sec4} we
consider $p$-constant characters for finite groups of Lie type in characteristic $p$ and prove Theorems 
\ref{th1} and \ref{th4}. In
Section \ref{Sspor} we analyse the sporadic groups and finally in Section \ref{sec6} we
prove Theorem \ref{main}.

\section{Blocks with cyclic defect group}\label{sec2}

We first make the following observation for an arbitrary finite group $G$. 
Let $p$ be a prime dividing the order of $G$ and let $S$ be a Sylow $p$-subgroup of $G$. 
Let $\ZZ$ denote the set of rational integers.

\begin{lemma}\label{21}
Let $\tau$ be a generalized character of a group $G$ such that
$\tau(u)=a$ for some complex number $a$ and every $1\neq u\in S$. 
Then $a\in\ZZ$.
\end{lemma}

\begin{proof}
Let $\lam\neq 1_S$  be an arbitrary linear character of $S$. Then
\begin{eqnarray*}
(\tau_{|S},\lam) & =& \frac{\sum_{u\in S} \tau(u)\lam(u)}{|S|}=\frac{\tau(1)+\sum_{1\neq u
\in S
} a\lam(u)}{|S|}\\
& =& \frac{\tau(1)+a\cdot |S|\cdot
(\lam,1_S)-a}{|S|}=\frac{\tau(1)-a}{|S|},
\end{eqnarray*}
since  $(\lam,1_S)=0.$
Hence, $a=\tau(1)-|S|\cdot (\tau_{|S},\lam) \in \ZZ$.
\end{proof}

The following lemma reduces Problem \ref{prob1} to groups with trivial center,
in particular, we can ignore quasi-simple groups that are not simple.

\begin{lemma}\label{cov}
Let $G$ be a finite group, $p$ a prime and let $\chi$ be 
an  irreducible character of $G$ of non-zero
$p$-defect.
Suppose that $\chi$ is non-trivial and  $p$-constant. Then  
one of the following holds:
\begin{itemize}
\item[(1)] $p$ does not divide $|Z(G)|$,  $Z(G)\leq\Ker(\chi)$  and the corresponding character $\bar \chi$
of $G/Z(G)$ is an irreducible $p$-constant character;
\item[(2)]
 $p=2$,  $|Z(G)|=2$, the order of $\Ker(\chi)$ is odd and $G=\Ker(\chi)\times Z(G).$
\end{itemize}
\end{lemma}

\begin{proof} 
Suppose that $Z(G)$ is not contained in $\Ker(\chi)$, and let $z\in Z(G)$, $\chi(z)\neq \chi(1)$.
If $z$ is not a $p$-element then $zg\in
\varSigma_p(G)$ for every $p$-element $g\in G$. So $ \chi(g)=\chi(zg)=\frac{\chi(z)}{\chi(1)}\chi(g)$. As
$\chi$ is not of $p$-defect 0, we have $\chi(g)\neq 0$, and hence $\chi(z)=\chi(1)$, which is a contradiction.

It follows that $Z(G)$ is a $p$-group, so $z\in
\varSigma_p(G)$. By Lemma
\ref{21}, $\chi(z)\in \ZZ$ and so $\chi(z)=- \chi(1)$, whence $p=2$ and $\Ker(\chi)$ has odd order. 
Let $h\in G$. Suppose that 
$zh\in
\varSigma_2(G)$. Then $-\chi(1)=\chi(z)=\chi(zh)=\frac{\chi(z)}{\chi(1)}\chi(h)=-\chi(h).$ So $h\in  \Ker(\chi)$. It follows that $G/\Ker(\chi)$ is a 2-group, and hence, for $g\in G$,   either $zg\in \varSigma_2(G)$ or $g\in z\cdot \Ker(\chi)$. In the former case $g\in \Ker(\chi)$, so 
$G=\Ker(\chi)\cup z\Ker(\chi)$, and hence $G=\Ker(\chi) \times \langle z \rangle$.
\end{proof}

Prior  proving  Theorem \ref{th2} we recall certain facts from representation
theory of
groups with cyclic Sylow $p$-subgroups. For further details, see \cite{Da}.
Let $G$ be a finite group with cyclic Sylow $p$-subgroup $S$. Set $C=C_G(S)$, 
$N=N_G(S)$ and $n=|N:C|$. As $(n,p)=1$, it follows that $n$ divides $p-1$.
  
Let $B$ be a block of $G$ having defect group $S$. By Brauer's first main theorem, there exists
a unique block $B_0$ of $N$ with the same defect group $S$ such that $B_0^G=B$. Let $b_0$ be a
block of $C$ such that $b_0^N=B_ 0$ (also  $S$ is the defect group of $b_0$). Let $E$ be the
subgroup of $N$ fixing $b_0$ and $e=|E:C|$. Then   $E/C$ acts on $S$ as a group of
automorphisms and  $e$ divides $p-1$ 
 ($e$ is called the inertia index of $B$).

The set of non-trivial irreducible characters  of $S$ 
partitions  into $(|S|-1)/e$ orbits under the action of $E/C$. Each of these orbits
contains  $e$ elements. Let $\Lambda$ be a complete set of representatives of these
classes. So $d=|\Lambda|=(|S|-1)/e$.

For a non-trivial character
$\lam\in\Irr S$ let $\eta_\lam$ denote the sum of all
$N/C$-conjugates of $\lam$.
In particular, $\eta_\lam(1)=n$
and $(\eta_\lam,1_S)=0$.
Note that if  $\lam,\mu$ are $N$-conjugate, then $\eta_\lam=\eta_\mu$.

\begin{lemma}\label{da1}\cite[Theorem 1 and Corollary 1.9]{Da} 
Under the previous hypothesis on $G,S,B,\ldots$, the block $B$ contains 
$e$ non-exceptional characters $\chi_1,\ldots , \chi_{e} $ and 
$d=|\Lambda|$ exceptional
characters $\chi_{\lam}$ $(\lam\in \Lambda)$.
Let $g\in S$ be of order $|S|$ and let $\phi$ be the unique irreducible Brauer character
of $C$ contained in $b_0$. 
When $|S|>p$, let $x\in S$ be of order $p$ and
$S_1=\langle x \rangle$. Set $N_1=N_G(S_1)$ and $C_1=C_G(S_1)$.
\begin{itemize}
\item[(1)] For any $j=1,\ldots , e$ one has 
$$\chi_j(g)=\ep_j\phi(1)\cdot  |N:E|\quad 
\textrm{ and } \quad \chi_j(x)=\ep_j\gamma \phi_1(1)\cdot |N_1: EC_1|,$$
for some $\phi_1\in \IBr C_1$,
$\ep_j=\pm 1$ and $\gamma=\pm 1$ that do not depend on $g$ and $x$.
\item[(2)] For any $\lambda \in \Lambda$ one has 
$$\chi_{\lam}(g)=\ep_0 \phi(1) \eta_\lam(g) \quad \textrm{ and
 }\quad \chi_\lambda(x)=\ep_0\gamma \phi_1(1)\sum_{ y \in
N_1/C_1}\lambda^y(x),$$
where $\phi_1,\gamma $ are the same as in item {\rm (1)} and $\ep_0=\pm 1$ does not depend
on $g,x$ and $\lambda$.
\item[(3)] $\chi_\lambda(S)\subset \Q$
if, and only if, $n=p-1$ and $|\lambda (S)|=p$.
\item[(4)] Assume $B$ is the principal block of $G$ and $d>1$.
Then the trivial character
$1_G$ is not exceptional,
except possibly when $p=2$ and 
$G$ has a normal subgroup of index $|S|$.
\end{itemize}
\end{lemma}

\begin{proof}  
(3) and (4) are  not stated in \cite{Da}, so we provide a proof here, although
they can be known to some experts.

(3) Let $S=\langle g\rangle$. By item (2), 
$\chi_\lam(g)=\ep_0 \phi(1)\eta_\lambda(g)$, where $\eta_\lambda(1)=n$. Let $\rho$ be a 
representation of $S$ with character $\eta_\lam$.  
Let  $|\lambda (S)|=p^a$
for some integer $a>0$ (so  
$\lambda(g)$
is a primitive $p^a$-root of unity). 
Let $\alpha_1\ld \alpha_n$ be the eigenvalues of $\rho(g)$. As $\eta_\lam$ is the sum of
all
$N$-conjugates of $\lam$ and $1_S\neq \lam\in \Irr S$, it follows that  $\alpha_1\ld
\alpha_n$ are  (distinct) roots
of the polynomial $t(x):=(x^{p^a}-1)/(x^{p^{a-1}}-1)$, which is irreducible over $\Q$.
Let $f(x)$ be the characteristic polynomial of the matrix $\rho(g)$, so $\alpha_1\ld
\alpha_n$
are also the roots of $f(x)$.
Suppose that that $\chi_\lambda(h)$ is rational for every $h\in S$.
Then so is $\eta_\lambda (h)$ and $\eta_\lambda(h)=\alpha^k_1+\cdots 
+\alpha^k_n\in\Q$ for every
integer $k$.
It is well known that the coefficients of the  characteristic polynomial of a square 
matrix $M$, say, are polynomials 
of the traces of $M^i$ for various integers $i$. It follows that the coefficients of $f(x)$ are rational.
This implies that the  polynomial $t(x)$ is reducible over $\Q$,
unless $t(x)=f(x)$. In the latter case all primitive $p^a$-roots are the roots of $f(x)$.
Therefore, $n=p-1$ and $a=1$. 

The converse is obvious.

\smallskip
(4)  Suppose the contrary, that $1_G$ is an exceptional character (this belongs to the
principal block).
We first recall that the principal blocks of $G$ and $N$ correspond to each other
under
the Brauer correspondence \cite[61.16]{CR2}. Furthermore,
it follows from 
by \cite[61.7 and 61.11]{CR2}, that  $1_N$ is the only irreducible Brauer character in
the principal block of $N$. 
In particular, $\phi$ is the trivial Brauer character of $C$. The group $E$ 
above is in fact the stabilizer in $N$ of this character, and hence   $E=N$, $e=n$ in this case.

By item (3),  $n =p-1$.  Let $1_G=\chi_{
\lambda}$ for some $\lambda\in \Lambda$ such that $| \lambda(S)|=p$.
Therefore, $\eta_\lambda(g)=-1$. 
Since $d>1$ we have $|S|>p$. 
Clearly,  
$|N_1/C_1|\leq p-1$. As $N\subseteq N_1$ and $|N/C|=p-1$, we have $|N_1/C_1|= p-1$.
By item (2) we get $1=\chi_{\lambda}(x)= \pm \phi_1(1)\cdot |N_1:C_1|$ for some
$\phi_1\in \IBr C_1$, whence $n=1$ and $p=2$.

The statement on the structure of $G$  follows from the  Burnside Normal Complement Theorem
\cite[Theorem 14.3.1]{Ha}.
\end{proof}

\begin{proof}[Proof of Theorem {\rm \ref{th2}}]
Clearly we may assume $\chi\neq 1_G$.

(1) It follows  from Lemma \ref{21} that $\chi(S)\subset \mathbb{Z}$,
so $\chi-a\cdot 1_G$ (where $a=\chi(s)$, $1\neq s\in S$) is a non-zero 
generalized 
character vanishing at the $p$-singular elements. It
follows from
\cite[Ch.IV, Lemma 3.14]{Fe} that  $\chi$ and $1_G$ belong to the same block. As $1_G$ is
in the principal block, so is $\chi$.

(2) By (1),
$1_G$ and $\chi$  belong to the principal block, and,  by assumption,  $\chi\neq 1_G$. 
Consider the Brauer tree associated to the principal block. Recall that one node of the
Brauer tree corresponds to the  sum of all $d=|\Lambda|$ exceptional
characters (denoted by $\chi_0$), and the
other nodes are in bijective correspondence with the $e$ non-exceptional characters of
the
block.  

(i) The theorem is true if   both $\chi$ and $1_G$  are  not exceptional  or $d=1$.

Let $v,w$ be the nodes at the Brauer tree corresponding to the characters $1_G$ and
$\chi$,  let
$n_1=v, n_2,\ldots, n_k=w$ be the consequent nodes of the path connecting
$v$ and $w$, and let $\psi_i$ be the ordinary character corresponding to $n_i$ for
$i=1,\ldots, k$ (one of the characters $\psi_i$ coincides with $\chi_0$, which is 
irreducible if and only if
$d=1 $). 
By  \cite[Ch.VII, Lemma
2.15]{Fe}, $\psi_i+\psi_{i+1}$ is the character of a projective indecomposable module
for $i=1,\ldots ,k-1$. Let $g\in \varSigma_p(G)$. Then 
$\psi_i(g)=-\psi_{i+1}(g)$ for every $i=1,\ldots ,k-1$. It follows that
$\psi_i(g)=(-1)^{i+1}\psi_1(g)$. As $\psi_1=1_G$, we arrive at the case (2)(a).
In addition, this proves the additional statement of the theorem in this case. 

(ii) The theorem is true if $d>1$ and $\chi$ or $1_G$ is exceptional.

Let  $d>1$. If $1_G$ is exceptional then, by Lemma \ref{da1}(4), $n=e=1$, $p=2$ and
$G/P\cong S$ for a normal subgroup $P$
of $G$.  We show that the same is true if $1_G$ is non-exceptional.

Let $\chi=\chi_\lam$ for some $\lam\in\Lambda$.
By Lemma \ref{da1}(3), $\chi_\lambda(S)\not \subset \Q$ unless
$n=p-1$ and $|\lambda(S)|=p$. In this case $\eta_\lam(g)=-1$ and $\chi(g)=\pm 1$.
Let $x\in S$ be of order $p$ and
$S_1=\langle x \rangle$. Set $N_1=N_G(S_1)$ and $C_1=C_G(S_1)$.  As $\lam(x)=1$, by  Lemma \ref{da1}(2),
we have $\chi_{\lambda}(x)=
\pm \phi_1(1)\cdot |N_1:C_1|$ for some $\phi_1\in \IBr C_1$. As $\chi$ is
$p$-constant,   $\chi(g)= \chi(x)$, whence $N_1=C_1$.  It is well known 
$N\cap C_1=C$. As $N\subseteq N_1$, we have $N=C$, and hence $n=1$. Then $n=p-1$ implies
$p=2$.  

By the Burnside Normal Complement Theorem
\cite[Theorem 14.3.1]{Ha}, $G$ has a normal $2$-complement,
that is, $G$ has  a normal subgroup $P$, say, of index $|S|$ as claimed.

Furthermore,  $\chi$ belongs to the principal block $B$, so the irreducible constituents of 
$\chi|_P$ are in the principal block $b$ of $P$, see \cite[Theorem 9.2]{Na}. 
As $P$ is a $p'$-group,
$1_P$ is the only irreducible character in $b$. Therefore, $\chi|_P=\chi(1)\cdot 1_P$,
and hence $P$ is in the kernel of  $\chi$. 
So $\chi$ is linear.  
Since $d>1$, necessarily $|S|>2$.
In this case, the hypothesis $\chi$ be $p$-constant leads to the contradiction
$\chi=1_G$.  

To prove the converse, suppose that $\chi$ is non-exceptional. 
By Lemma \ref{da1}(1), $\chi(S)\subset \Q$.  By (i), we only have to deal 
with the case where $1_G$ is exceptional. Then by Lemma \ref{da1}(4), $n=e=1$ and
$G/P\cong S$ for a normal subgroup $P$
of $G$. Then we have seen in the previous paragraph that
$\chi(P)=1$. 
Now $\chi(S)\subset \Q$ implies $\chi^2=1_G$. Then $\chi(g)= \chi(g^2)$ leads to $\chi = 1_G$
which is a contradiction.
\end{proof}

 It is well known that a defect group of the
principal block of $G$ coincides with a Sylow $p$-subgroup.
Therefore, if a $p$-constant character 
belongs to a block with cyclic defect group then, by Theorem \ref{th2}(1), the
Sylow $p$-subgroups of $G$ are cyclic.

\section{Symmetric and alternating groups}\label{sec3}

We first consider the case where Sylow $p$-subgroups of $G=\Sym_n$ are cyclic,
equivalently with $p\leq n<2p$.
By Theorem \ref{th2}, a character $\tau\in\Irr G$ of non-zero defect is $p$-constant if,
and only if, $\tau$ belongs to the principal block. 
Therefore, it suffices to determine the
non-linear irreducible characters that are in the same block as $1_G$.
However, this is already  known, see \cite[6.1.21]{JK}. 
Specifically, if $\chi_\lambda$ is the irreducible character of $G$ corresponding to a
partition $\lambda$ of $n$, then $\chi_\lambda$ is in the principal block if and only if
the $p$-core of $\lambda$ is the same as that of the trivial partition $(n)$. (See
\cite[p.76]{JK} for the notion of $p$-core.)
If $n=p$ then the $p$-core of $(n)$ is empty; this implies that $\lambda$
is a hook. If $p<n<2p$ then the $p$-core of
$(n)$ is $(n-p)$. It follows that 
$\lambda$ is the  partition associated to the diagram obtained from a hook
diagram  associated to a partition $\lambda'\neq (p)$ for $\Sym_p$
either by adding   $(n-p)$ boxes to the second row, or by adding the additional row of
$(n-p)$ boxes   above  the diagram of $\lambda'$,
 provided this yields a proper
diagram. In more accurate terms this is described in  the following lemma.

\begin{lemma}\label{cyc Sn}
Let $p$ be a prime such that $2\leq p\leq n< 2p$. A non-linear irreducible character
$\chi_\lambda$ of $\Sym_n$ of non-zero defect is $p$-constant if, and only if,
$n$ and $\lambda$ satisfy one of the following conditions:
\begin{itemize}
\item[(i)] $n=p\geq 3$ and $\lambda=(b,1^{p-b})$ with $2\leq b\leq p-1$;
\item[(ii)] $n=p+1\geq 4$ and  $\lambda=(b+1,2,1^{p-b-2})$ with $1\leq b\leq  p-2$;
\item[(iii)] $n=p+r\geq 5$, $r\geq 2$,  and $\lambda$ is one of the following partitions:
$$(p-a, r+1, 1^{a-1})\;\; (1 \leq  a  \leq  p-r-1);\qquad 
(r, b, 1^{p-b})\;\; (1\leq b\leq r).$$
\end{itemize}
\end{lemma}

Note that if $n=p,p+1$ then $G$ has a single block of non-zero defect \cite[86.10]{CR}.
We consider now the alternating groups.

\begin{propo}
Let $p$ be a prime such that $2<p\leq n< 2p$.  A non-linear irreducible character
$\tau$ of $G=\Alt_n$ of non-zero defect is $p$-constant if, and only if,
 $\tau$ is a
constituent of $\chi_{\lambda}|_G$, where $\chi_\lambda\in \Irr \Sym_n$ and  one
of the
following holds:
\begin{itemize}
\item[(i)] $n=p\geq 5$ and $\lambda=(b,1^{p-b})$ with $2\leq b\leq p-1$ and $b\neq
\frac{p+1}{2}$;
\item[(ii)] $n=p+1\geq 6$ and $\lambda=(b+1,2,1^{p-b-2})$ with $1\leq b\leq  p-2$ and
$b\neq
\frac{p-1}{2}$;
\item[(iii)] $n=p+r$, $r>2$, and $\lambda$ is one of the following partitions:
$$(p-a, r+1, 1^{a-1})\;\; (1 \leq  a  \leq  p-r-1);\qquad 
(r, b, 1^{p-b})\;\; (1\leq b\leq r).$$
\end{itemize}
\end{propo}

\begin{proof}
By \cite[Theorem 6.1.46]{JK}, the characters of the principal $p$-block of $\Alt_n$ are
 constituents of the characters $\chi_\lambda$, where $\lambda$ is  one of
the partitions described in Lemma \ref{cyc Sn}. Denote by $\lam^T$ the partition
associated
with the diagram transpose to that of $\lam$.
If $\lambda\neq \lambda^T$, the restriction $\tau=\chi_\lambda|_G$ is irreducible and so
$\tau$ is $p$-constant.
Consider now the case $\lambda=\lambda^T$. This happens only for $n=p$ when
$\lambda=(\frac{p+1}{2}, 1^{\frac{p-1}{2}})$
and for $n=p+1$ when  $\lambda=(\frac{p+1}{2},2, 1^{ \frac{p-3}{2}})$. 
In these cases the group $\Alt_n$ has two conjugacy classes $\sigma_+,\sigma_-$ of elements of
order $p$.
For the previous values of  $\lambda$, the character $\chi_\lambda$
splits on $\Alt_n$ as two characters $\tau_1,\tau_2$.
By \cite[Theorem 2.5.13]{JK} we have $\tau_i(\sigma_{\pm})=
\frac{ (-1)^{ (p-1)/2} \pm \sqrt{p(-1)^{
(p-1)/2}}}{2}$, whence these  $\tau_i$'s are not $p$-constant.
\end{proof}

Now, suppose that $n\geq 2p$. From the proof of Propositions 4.2 and 4.3 of
\cite{end} we can deduce the following result.

\begin{lemma}\label{end An}
Let $\chi_\lambda$ be the  irreducible character of $H=\Sym_n$ associated
to the partition $\lambda$.
Let $\tau_\lambda$ be an irreducible character of $G=\Alt_n$ which is a
constituent of  $\chi_\lambda|_G$. Let $p>2$ be a prime
such that $n\geq 2p$. Then 
$\chi_\lambda(h)=0$ for some  $h\in \varSigma_p(H)$, unless possibly when $\lambda$ is
conjugate to one of the
following partitions:
\begin{itemize}
\item[(i)] $2p\leq n=2p+r\leq 3p-1$ and $\lambda=(p+r, r+1, 1^{p-r-1})$;
\item[(ii)] $n=2p$, $\lam=(p,2,1^{p-2})$;
\item[(iii)] $n=2p+1$, $\lam=(p+1,1^p)$.
\end{itemize}
Similarly, $\tau_\lambda(g)=0$ for some  $g\in \varSigma_p(G)$, unless possibly when
$\lambda$ is one of the partitions of items {\rm (i)} to {\rm (iii)}.

Let $p=2$. If $n>11$ then every non-linear irreducible
character of $\Sym_n$ and every non-linear irreducible
character of $\Alt_n$ vanishes at some $2$-singular element.
\end{lemma}

To deal with the missing cases of the previous Lemma, we look at the character
table of $G=\Sym_n,\Alt_n$, when $n\leq 11$, obtaining the
following irreducible  characters $\tau$ of $G$ that do not vanish at
$\varSigma_2(G)$:
\begin{eqnarray*}
\Sym_n: & (n,\tau(1))\in  & \{ (4,3), (5,5)\};\\
\Alt_n: & (n,\tau(1))\in  & \{(4, 3), (5,3),(5,5), (6,5), (6,9), (7,15),(7,21),(7,35),\\
&& (10, 315), (11, 165)\}.
\end{eqnarray*}

However, among these characters, only those described in Theorem \ref{th3}(3)  and
the irreducible character of degree $3$ of $\Alt_4$  are $2$-constant.

\smallskip
As an application of Murnaghan-Nakayama formula (e.g., see \cite[2.4.7]{JK})  we
prove the
following.

\begin{propo}\label{pr2}
Let $p$ be a prime such that $ n\geq  2p\geq 4$. Then the non-linear 
$p$-constant irreducible
characters of $\Sym_n$ are all of $p$-defect $0$.
\end{propo}

\begin{proof}
Let, as before, $\chi_\lambda$ be the irreducible character of $\Sym_n$ associated to the
partition $\lambda$ of $n$.
Assume that $\chi_\lambda$ is of non-zero defect.
When $p=2$,  it suffices to look at the character tables for the cases $n\leq 11$
as done before. So, suppose $p>2$. By Lemma \ref{end An}  we are left to
consider the case  $n=2p+r$ ($0\leq r<p$).

First, take $\lambda=(p+r, r+1, 1^{p-r-1})$ with $0\leq r< p$. We
apply 
Murnaghan-Nakayama formula to permutations $\sigma$ whose cyclic decomposition is of type
$(2p)(r)$ or $(p+r)(r)$, obtaining $\chi_\lambda((2p)(r))=(-1)^{r+1}$ and
$\chi_\lambda((p+r)(p))=(-1)^r$. When $\lambda=(p+r, r+1, 1^{p-r-1})^T$, we obtain 
$\chi_\lambda((2p)(r))=-1$ and $\chi_\lambda((p+r)(p))=+1$.
This means that these characters $\chi_\lambda$ are not constant on $\varSigma_p(G)$.
Now, if $n=2p$ and $\lambda=(p,2,1^{p-2})$, then $\chi_\lambda((2p))=0$.
Finally, if $n=2p+1$ and $\lam=(p+1,1^p)$, then $\chi_\lambda((2p)(1))=0$.
\end{proof}

\begin{proof}[Proof of Theorem {\rm \ref{th3}}]
If $n\geq 2p>4$, by Lemma \ref{end An} we are reduced to the following cases:
\begin{itemize}
\item[(a)] $n=2p$ and $\lambda=(p,2,1^{p-2}), (p,1^{p})$;
\item[(b)] $n=2p+1$ and $\lambda=(p+1,1^p), (p+1,2,1^{p-2})$;
\item[(c)] $n=2p+r$, $r>1$, and $\lambda=(p+r,r+1,1^{p-r-1})$.
\end{itemize}
Actually, we can exclude case (c). Using Murnaghan-Nakayama formula we obtain 
$\chi_\lambda((p+r-1)(p)(1))=(-1)^r$ and  $\chi_\lambda((2p)(r-1)(1))=(-1)^{r+1}$.

The case $p=2$ follows from Lemma \ref{end An} and previous direct computations
for $n\leq 11$.
\end{proof}

\section{Groups of Lie type}\label{sec4}

Following \cite[1.17]{Ca} we use the term ``a group of Lie type'' to refer to groups of
shape  $\G^F$, where $\G$
is a connected reductive algebraic group in defining characteristic $p$ with an
algebraic group
endomorphism
$F:\G \ra \G$ such that  the subgroup $\G^F:=\{g \in \G: F(g)=g\}$
is finite. Such an endomorphism is called a Frobenius map ($F$ is not
necessarily the standard Frobenius map).
In what follows $\G$ is assumed to be simple, not necessarily simply connected. 

In Lemma \ref{22} below the term ``regular character'' is used as in the Deligne-Lusztig
theory.
More precisely, a regular character is defined to be a constituent of a  Gelfand-Graev
character
\cite[14.39]{DM}, where the latter is the induced character $\lam^G$ when $\lam$
is a linear
character of a Sylow $p$-subgroup $U$
satisfying a certain non-degeneracy condition. Every group of Lie type has at least one  
Gelfand-Graev character. In addition, every Gelfand-Graev character is multiplicity  free
and does not have $1_G$ as a constituent.

\begin{lemma}\label{22}
Let $\G$ be a connected reductive group defined over a field of characteristic
$p$, $F$ a Frobenius endomorphism and  $G=\G^F$. 
Let $U$ be a Sylow $p$-subgroup of $G$ and let  $\tau$ be an irreducible
character of $G$ such that $\tau(u)=a\neq 0$ for all $1\neq u\in U$.
Then either $ \tau(1)=1$  or $\tau$ is regular, $a=\pm 1$ and $\tau(1)=a+|U|$.
\end{lemma}

\begin{proof} 
By Lemma \ref{21}, $a\in \ZZ$ and so $\chi=\tau-a\cdot 1_G$ is a $Syl_p$-vanishing
generalized character of $G$ (i.e. vanishing on $U\setminus\{1\}$). If 
$\chi(1)=0$ then $\tau(1)=a$, and  hence $U \leq \Ker(\tau)$. It follows that that the
normal 
subgroup $X$ of $G$ generated by the unipotent elements is contained in $\Ker(\tau)$. 
It is well known that $G/X$ is abelian, and hence $\tau(1)=1$. 

Suppose $\chi(1)\neq 0$.
Then $\chi(1)$ is a multiple of $|U|$ (cf. \cite{PZ}). 
Observe that for any linear character $\lambda$ of $U$ we have
$$(\chi,\lambda^G)=(\chi_{|U},\lambda)=\frac{\chi(1)}{|U|}=(\tau,\lambda^G)-a\cdot (1_G,
\lambda^G).$$ 
In particular, $\tau$ is a regular character of $G$. 
As $\lambda^G$ is multiplicity  free,  we have $(\tau,
\lambda^G)=1$, whence $\chi(1)=|U|$.
Furthermore, $\chi(1)=|U|$ implies that $(\chi,\lambda^G)=1$ for any linear character
$\lambda$ of $U$.
In particular, taking $\lambda =1_U$ we obtain
$1=(\tau,1_U^G)-a$, whence $(\tau,1_U^G)=a+1\geq 0$ and  $a\geq -1$. 
 We show that $a=\pm 1$.

We first consider the case where $Z(\G)$ is connected.
Since $\tau$ is constant on $U\setminus \{1\}$, the average value
of $\tau$ on any set of regular unipotent elements of $G$ coincides with its value $a$. 
Hence, by  \cite[Theorem 8.3.3(i)]{Ca}, we have $a=\pm 1$ (as $a\neq 0$).

Next, suppose that $Z(\G)$ is not connected. Note that $\G$ can be embedded in a reductive
group
$\hat \G$ with connected center such that the derived groups 
$\hat \G'$ and $\G'$ coincide.  Moreover, each Frobenius endomorphism of $\G$
extends to that of $\hat\G$ \cite[pp. 139-140]{DM}. 
We keep $F$ to denote the extended endomorphism of $\hat\G$. 
Then $G=\G^F \leq \hat{G}=\hat\G^F$; moreover $G$
is a normal subgroup of $\hat G$ with abelian quotient ({\it loc.cit}.). 
Let $\sigma$ be an irreducible constituent of $\tau^{\hat G}$. By Clifford's theorem,
$\sigma_{|G}=e\sum_i^t \tau_i$, where $\{\tau_1=\tau,\tau_2 ,\ldots,\tau_t\}$ are the
distinct
conjugates of $\tau$ and $e=(\sigma,\tau^{\hat G})$.
Since $\tau$ is constant on the set of the non-trivial unipotent elements of $G$, so are
all the
 $\tau_i$'s, and moreover, $\tau_i(u)=a$ for every $1\neq u\in U$. This means 
that also $\sigma$ is constant on $U\setminus \{1\}$, and, in addition, 
$\sigma(u)=et\cdot \tau(u)\neq 0$ is an integer. 
By the above, $\sigma(u)=\pm 1$, whence $et=1$ and so $a=\tau(u)=\pm 1$.  
\end{proof}

In the proof of the following two lemmas, we will make use of the Zsigmondy primes.
Here, we briefly recall their definition.
Let $a,n$ be two positive integers. If $a\geq 2$, $n\geq 3$ and $(a,n)\neq (2,6)$, then there exists a prime, denoted here by 
$\zeta_n(a)$, dividing $a^n-1$ and coprime to $a^i-1$ for every $1\leq i<n$. This prime, not necessarily unique, is called a Zsigmondy prime (or a primitive prime divisor of  $a^n-1$).  
Observe that if $\zeta_n(a)$ divides $a^k-1$, then $n$ divides $k$.
 
\begin{lemma}\label{-1a}
Let $\G$ be a simple connected reductive group and let $G=\G^F$ be the
corresponding finite group. Then $|G|_p-1$ divides $|G|$ if, and only if, $G\in \{A_1(q),
A_2(q), A_3(2), B_2(q), C_2(q),  G_2(q) \}$.
\end{lemma}

\begin{proof}
First, consider the groups of type ${}^ 2B_2(q^ 2)$ and ${}^ 2G_2(q^ 2)$. If $G={}^2B_2(q^2)$, where $q^ 2=2^{2n+1}$, then $|G|_2-1=q^4-1$ does not divide
$|G|=q^4(q^2-1)(q^4+1)$, as $\gcd(q^2+1,q^ 4+1)=1$. If $G={}^2G_2(q^2) $, where $q^ 2=3^ {2n+1}$, then $|G|_3-1=q^6-1$ does not divide
$|G|=q^6(q^2-1)(q^6+1)$, as  $\gcd(q^6-1,q^6+1)=2$.

Now, let $|G|_p=q^m$, with $G\not \in \{  {}^ 2B_2(q^ 2),{}^ 2G_2(q^ 2)\}$. 
We  start our analysis with the cases where the existence of a Zsigmondy prime $\zeta_m(q)$ is not guaranteed, i.e. $m\leq 2 $ or $(m,q)=(6,2)$. 
If $m\leq 2$, then  $G=A_1(q)$. In this case, $|G|_p-1=q-1$ divides $|G|=q(q^2-1)$.
If $(m,q)=(6,2)$ then $G$ is one of the following groups:
$A_3(2)$, $^2A_3(2)$, $G_2(2)$.
In this case, we can directly check  when $|G|_p-1$ divides $|G|$. This happens only when $G=A_3(2), G_2(2)$.

Hence, we may assume $m\geq 3$ and $(m,q)\neq (6,2)$. Under this assumption, a Zsigmondy prime $\zeta_m(q)$
exists, and we check when this prime divides $|G|$. We show  that this happens only for the groups of rank $2$ in the statement.

If $G$ is of type $A_n(q)$, then  $m=\frac{n(n+1)}{2}\geq 3$. Suppose that
$\zeta_m(q)$ divides $|G|$. Then $\frac{n(n+1)}{2}\leq n+1$, whence $n^2-n-2\leq 0$ and
so $n=2$. In this case, $|A_2(q)|_p-1=q^3-1$ divides $|G|$.
If $G$ is of type ${}^2 A_n(q)$, then $m=\frac{n(n+1)}{2}\geq 3$. Suppose that $\zeta_m(q)$ divides  $ |G|$. 
Then $\frac{n(n+1)}{2}\leq 2(n+1)$, whence  $n=3,4$. For $n=3$,  $|{}^ 2A_3(q)|_p-1=q^3-1$ does not divide $|G|=q^3(q^3+1)(q^2-1)$ since
 $\gcd(q^3-1,q^3+1)\leq 2$. For $n=4$, $|{}^ 2A_4(q)|_p-1=q^6-1$ and $|G|=q^6(q^4-1)(q^3+1)(q^2-1)$. Notice that $\zeta_3(q)$ does not divide 
 $(q^ 4-1)(q^2-1)$, whence $|G|_p-1$ does not divide $|G|$.

If $G$ is of type $B_n(q)$ or $C_n(q)$, then $m=n^2\geq 3$. The condition $\zeta_m(q)$ divides  $|G|$ implies  $n^2\leq 2n$ and so $n=2$. On the other hand, if $n=2$
then $|G|_p-1=q^4-1$ divides  $ |G|=q^4(q^4-1)(q^2-1)$.
If $G$ is of type $D_n(q)$, then $m=n^2-n\geq 3$. 
The condition $\zeta_m(q)$ divides  $ |G|$ implies $n^2-n\leq 2n-2$ and so $n=1,2$.
If $G$ is of type ${}^2D_n(q)$, then $m=n^2-n\geq 3$. If
$\zeta_m(q)$ divides  $ |G|$ then $n^2-n\leq 2n$ and so $n=3$. In this case, $|G|_p-1=q^6-1$ and $|G|=q^6(q^4-1)(q^2-1)(q^3+1)$. However,
the prime $\zeta_3(q)$ divides $|G|_p-1$, but does not divide $|G|$.

Similarly, if $G$ is of type $E_6(q)$, ${}^2E_6(q)$, $E_7(q)$,  $E_8(q)$ or
$F_4(q)$, then it is straightforward
to see that  $|G|_p-1$ does not divide $|G|$.
If $G$ is of type ${}^2F_4(q^2)$, where $q^ 2=2^{2n+1}$, then $\zeta_{6n+3}(2)$ divides $|G|_2-1$ but does not divide
$|G|$.
If $G$ is of type ${}^3D_4(q)$, then $m=12$. In this case,
$|G|=q^{12}\frac{(q^6-1)(q^{12}-1)}{(q^2+1)}$. Suppose that $|G|_p-1=q^{12}-1$ divides
$|G|$, then we obtain that $q^2+1$  divides  $q^6-1$, i.e. $\zeta_4(q) $ divides $q^6-1$,
which does not happen, since $4$ does not divide $6$. Finally,  if $G$ is of type
$G_2(q)$ then $|G|_p-1$ divides $|G|$.
\end{proof}

With the same techniques used in the proof of   the previous lemma, we can also prove
the following one.

\begin{lemma}\label{+1a}
Let $\G$ be a simple connected  reductive  group and let $G=\G^F$ be the
corresponding finite group. Then $|G|_p+1$ divides $|G|$ if, and only if, $G\in
\{A_1(q), {}^2A_2(q), {}^2B_2(q^2),{}^2G_2(q^2) \}$.
\end{lemma}

As previously remarked, if $G$ is quasi-simple then the Steinberg character is the only irreducible character of $p$-defect $0$. 
So we can now prove Theorem \ref{th1}.


\begin{proof}[Proof of Theorem \emph{\ref{th1}}]
Assume that  $\tau(1)>1$ and  $\tau(u)=a\neq
0$.
By Lemma \ref{22}, $a=\pm 1$ and $\tau$ is a regular  character of degree $
|G|_p\pm 1$.
If $a=-1$, by Lemma \ref{-1a} we are reduced to consider the following groups:
$SL_2(q), SL_3(q), SL_4(2), Sp_4(q)$ and $G_2(q)$. Since
$\chi=\tau+
1_G$  is a proper $Syl_p$-vanishing character, we may use
\cite{PZ}.
If $a=1$,  by Lemma \ref{+1a}, it suffices to consider the following groups:
$SL_2(q), SU_3(q), {}^2B_2(q^2)$ and ${}^2G_2(q^2)$.
For all these groups the result follows by 
 analysis of the character tables.
\end{proof}

\begin{proof}[Proof of Theorem \emph{\ref{th4}}]
As $\tau(1)=|G|_p\pm 1$ and $\tau(1)$ divides $|G|$, it follows by Lemmas \ref{-1a} and \ref{+1a}
that $G$ must be one of the following groups: $SL_2(q)$, $SL_3(q)$, $SL_4(2), SU_3(q), Sp_4(q), {}^2B_2(q^2), G_2(q), {}^2G_2(q^2)$. 
So, it suffices to inspect the character tables of these groups 
to identify the irreducible characters with the degrees in question. 
\end{proof}

\section{Sporadic groups}\label{Sspor}

The answer to Problem \ref{prob1}  for quasi-simple
sporadic groups can be obtained directly from
their character tables. We describe here the most interesting properties.

\begin{propo}\label{spor}
Let $G$ be a finite quasi-simple sporadic group and let $p$ be a prime dividing $|G|$. 
Let $\Delta_p(G)$ be the set of the non-linear irreducible characters of $G$ whose
$p$-defect is not $0$ and which are constant on $\varSigma_p(G)$.
\begin{itemize}
\item[(1)] For every $G$ there exists a prime $p$ such that the set $\Delta_p(G)$ is not
empty.
\item[(2)] There exists exactly one prime $p$ such that $\Delta_p\neq \emptyset$ if, and
only if, $p=7$ and $G=J_2,2.J_2$.
\item[(3)] The set $\Delta_2(G)$ is not empty if, and only if, $G=J_1$. In this case
$\Delta_2(J_1)=\{\tau\}$ where $\tau(1)=209$.
\item[(4)] $G$ has a character $\tau\in\Delta_p(G)$ such that $|\tau(g)|\neq 1$ $(g\in
\varSigma_p(G))$ if, and only if, $p=3$, $G\in\{M_{22},  2.M_{22}, 4.M_{22}\}$ and
$\tau(1)=385$. In these cases $\tau(g)=-2$.
\end{itemize}
\end{propo}

Finally we consider the simple group $G={}^2F_4(2)'$, which admits the following
$p$-constant
irreducible characters (in the notation of \cite{ATLAS}):
\begin{itemize}
\item[(1)] $p=3$: $\chi_{8}$ of degree $325$;
\item[(2)] $p=5$: $\chi_9$ of degree $351$ and $\chi_{12},\chi_{13}$ of degree $624$;
\item[(3)] $p=13$: $\chi_4,\chi_5$ of degree $27$, $\chi_7$ of degree $300$, $\chi_{15}$
of degree $675$ and $\chi_{20}$ of degree $1728$.
\end{itemize}
In all these cases, the characters take value $\pm 1$ on $\varSigma_p(G)$.

\section{Proof of Theorem \ref{main}}\label{sec6}

Let $G$ be a finite simple group and let $\tau\in \Irr(G)$ such that $\tau(g)=c$ for all
$g \in \varSigma_p(G)$, where $p$ is a prime dividing $|G|$. 
By Lemma \ref{21}, $c\in \ZZ$ and   $c=0$ precisely when $\tau$ is of
$p$-defect $0$. Also,
by Theorem \ref{th2}(2)
$c=\pm 1$ when $G$ has a cyclic Sylow
$p$-subgroup.
So, assume that $\tau$ is not of $p$-defect $0$ and that the Sylow
$p$-subgroups of $G$ are not cyclic.
If $G$ is an alternating group, then by Proposition \ref{th3} it follows that $c=\pm
1$. If $G$ is a sporadic group, from Proposition \ref{spor} we get that either $c=\pm 1$
or $G=M_{22}$, $c=-2$,  $p=3$ and $\tau(1)=385$.
Finally, for  groups of Lie type of characteristic $p$, the result follows from  Lemma
\ref{22}. This proves Theorem \ref{main}.

\section*{Acknowledgements}

We want to thank the referee for careful reading the manuscript and making a number of useful 
suggestions for improving it.

\end{document}